\documentclass[11pt,letterpaper]{article}
\usepackage[T1]{fontenc}
\usepackage{lmodern}
\usepackage[margin=1.05in]{geometry}
\usepackage{amsmath,amssymb,amsthm,mathtools}
\usepackage{enumitem,xcolor}
\usepackage[numbers,sort&compress]{natbib}

\definecolor{paperlink}{RGB}{0,0,114}
\usepackage[colorlinks=true,allcolors=paperlink,linktoc=all,pdfstartview=Fit]{hyperref}
\usepackage[capitalize,nameinlink,noabbrev]{cleveref}
\hypersetup{pdftitle={Lp-form of the KNRS conjecture},pdfauthor={Yuqi Zhao}}
\setlist[enumerate]{leftmargin=2em,topsep=0.3em,itemsep=0.2em}
\numberwithin{equation}{section}
\theoremstyle{plain}
\newtheorem{theorem}{Theorem}[section]
\newtheorem{lemma}[theorem]{Lemma}
\newtheorem{fact}[theorem]{Fact}
\newtheorem{corollary}[theorem]{Corollary}
\newtheorem{proposition}[theorem]{Proposition}
\newtheorem{conjecture}[theorem]{Conjecture}
\theoremstyle{definition}
\newtheorem{remark}[theorem]{Remark}
\crefname{fact}{Fact}{Facts}
\crefname{remark}{Remark}{Remarks}
\crefname{conjecture}{Conjecture}{Conjectures}
\newcommand{\Aut}{\operatorname{Aut}}
\newcommand{\supp}{\operatorname{supp}}
\newcommand{\1}{\mathbf{1}}
\newcommand{\dd}{\,\mathrm{d}}
\newcommand{\Sub}{\operatorname{Sub}}
\title{\texorpdfstring{$L^p$}{Lp}-form of the KNRS conjecture}
\author{Yuqi Zhao\thanks{Email: \href{mailto:yuqi.zhao012@gmail.com}{\texttt{yuqi.zhao012@gmail.com}}.}}
\date{}
\begin{document}
\maketitle

\begin{abstract}
Let~$F$ be a graph with~$m>0$ edges and~$n$ non-isolated vertices. For every~$\rho$-locally dense graphon~$W$, we prove that~$\|W_F\|_q\ge\rho^m$ whenever~$q\ge\binom n2/m$, where~$W_F$ is the product of~$W$ over the edges of~$F$. The KNRS conjecture asks for this bound at~$q=1$. More generally, averaging over the automorphisms of an edge-transitive KNRS supergraph~$J$ yields the sufficient exponent~$e(J)/m$; for~$\rho>0$ and~$q>e(J)/m$, equality holds only for constant graphons. The proof averages over vertex relabellings via H\"older's inequality, in the spirit of Conlon and Lee. We also prove that replacing every edge of a Sidorenko-good graph by a fixed theta pattern preserves goodness for all branch lengths, and that applying the same substitution to a KNRS graph yields a regular-KNRS graph. For non-uniform substitutions, we establish goodness under the branch-count conditions of Im, Li, and Liu and determine when equality forces the kernel to be constant. Passing to the one-subdivision recovers their theorem for even branch lengths.
\end{abstract}

\section{Introduction}

A graphon~$W\colon[0,1]^2\to[0,1]$ is a symmetric measurable function. It is \emph{$\rho$-locally dense} if~$\int_{S\times S} W\ge\rho\lambda(S)^2$ for every measurable~$S\subseteq[0,1]$, where~$\lambda$ denotes Lebesgue measure; taking~$S=[0,1]$ shows in particular that its edge density~$\int W$ is at least~$\rho$. The conjecture of Kohayakawa, Nagle, R\"odl, and Schacht~\citep{KNRS2010} asks whether the constant graphon of value~$\rho$ minimizes every graph density under this local condition.

For a finite graph~$F$ and~$\mathbf x=(x_v)_{v\in V(F)}$, write
\begin{equation*}
    W_F(\mathbf x)=\prod_{uv\in E(F)} W(x_u,x_v)\,,
    \qquad t(F,W)=\int_{[0,1]^{V(F)}} W_F(\mathbf x)\dd\mathbf x\,,
\end{equation*}
where~$\dd\mathbf x$ is product Lebesgue measure. A graph~$F$ is \emph{KNRS} if~$t(F,W)\ge\rho^{e(F)}$ for every~$\rho\in[0,1]$ and every~$\rho$-locally dense graphon~$W$, where~$e(F)=|E(F)|$. This formulation is equivalent to the original asymptotic conjecture for finite graphs~\citep[Lemma~2.6]{BradacSudakovWigderson2024}.

\begin{conjecture}[Kohayakawa--Nagle--R\"odl--Schacht]\label[conjecture]{conj:knrs}
Every finite simple graph is KNRS.
\end{conjecture}

\paragraph{A moment bound for every graph.}
The KNRS conjecture concerns the first moment of~$W_F$. Here we ask what moment exponent~$q$ suffices for the same lower bound on its~$L^q$ norm. For~$q>0$, write~$W^{\circ q}(x,y)=W(x,y)^q$ for the pointwise power, so that~$t(F,W^{\circ q})=\int W_F^q$.

\begin{theorem}\label{thm:main}
Let~$F$ be a simple graph with~$m>0$ edges and~$n$ non-isolated vertices. For every~$\rho\in[0,1]$, every~$\rho$-locally dense graphon~$W$, and every~$q\ge\binom n2/m$,
\begin{equation*}
    \|W_F\|_{L^q([0,1]^{V(F)})}\ge\rho^m\,,
    \qquad\text{equivalently}\qquad t(F,W^{\circ q})\ge\rho^{qm}\,.
\end{equation*}
\end{theorem}

The exponent~$\binom n2/m$ comes from distributing the~$m$ edges of~$F$ uniformly over all~$\binom n2$ vertex pairs. After deleting isolated vertices, taking the geometric mean of~$W_F$ over all vertex permutations assigns exponent~$m/\binom n2$ to each pair. Writing~$C=\binom n2$, this geometric mean is~$W_{K_n}^{m/C}$. At~$q=C/m$ its $q$th power is exactly~$W_{K_n}$, and~$K_n$ is known to be KNRS~\citep{KNRS2010}. Replacing~$K_n$ by a sparser edge-transitive KNRS supergraph sharpens the exponent.

All graphs in this paper are loopless, and parallel edges are counted with multiplicity when multigraphs are allowed. We write~$\supp(F)$ for the simple graph obtained from~$F$ by keeping one copy of each edge. A KNRS graph~$J$ is \emph{KNRS-forcing} if, for~$0<\rho<1$, equality~$t(J,W)=\rho^{e(J)}$ for a~$\rho$-locally dense graphon~$W$ implies~$W=\rho$ almost everywhere.

\begin{theorem}[Edge-transitive comparison]\label{thm:comparison-intro}
Let~$F$ be a loopless multigraph with~$m=e(F)>0$, and let~$J$ be a simple edge-transitive KNRS graph on the same vertex set with~$\supp(F)\subseteq J$. If~$q\ge e(J)/m$, then every~$\rho$-locally dense graphon~$W$ with~$\rho\in[0,1]$ satisfies
\begin{equation*}
    t(F,W^{\circ q})\ge\rho^{qm}\,.
\end{equation*}
If~$q>e(J)/m$ or~$J$ is KNRS-forcing, then for~$0<\rho<1$ equality holds if and only if~$W=\rho$ almost everywhere.
\end{theorem}

For example, comparing a bipartite graph~$F$ with part sizes~$a,b$ against~$K_{a,b}$ gives the exponent~$ab/m$. For a multigraph, the threshold~$e(J)/m$ can even be less than~$1$, so~\cref{thm:comparison-intro} is stated directly as a moment inequality. In~\cref{sec:exponent}, we show that among exponents~$q\ge1$, every simple graph with at least one edge admits a minimum admissible exponent, which equals~$1$ precisely when the graph is KNRS.

Known examples of KNRS graphs include all complete multipartite graphs~\citep{KNRS2010}, odd cycles~\citep{Reiher2014}, and all unicyclic graphs as well as cycles with a chord~\citep{Lee2021}. Brada\v{c}, Sudakov, and Wigderson~\citep{BradacSudakovWigderson2024} found gluing operations preserving the KNRS property and studied equality and stability in the conjecture.

A related moment inequality appears in work of Conlon and Lee on Sidorenko's conjecture. Sidorenko~\citep{Sidorenko1993} conjectured that~$t(H,W)\ge(\int W)^{e(H)}$ for every bipartite~$H$ and every graphon~$W$, following an earlier supersaturation conjecture of Erd\H{o}s and Simonovits~\citep{ErdosSimonovits1984}; a recent preprint of OpenAI~\citep{OpenAI2026} gives a bipartite counterexample. We call a bipartite graph \emph{Sidorenko} if it satisfies this inequality for all graphons. For a bipartite graph~$H$ with parts~$A,B$, Conlon and Lee~\citep[Corollary~1.3]{ConlonLee2021} proved an~$L^p$ bound for the conditional density obtained by fixing the vertices in~$A$ and integrating over~$B$. Our moment is instead taken on the full product~$W_F$ over all vertex variables, though both arguments use H\"older's inequality to average over vertex relabellings.

\paragraph{Theta substitutions.}
Our second set of results concerns graphs formed by replacing each edge by internally disjoint paths between its endpoints. Throughout this discussion, assume that the base graph has at least one edge. A graph is \emph{Sidorenko-good} if it satisfies the Sidorenko inequality on all bounded nonnegative positive semidefinite kernels~\citep{SidorenkoDNN2021}. We prove that a uniform theta substitution preserves this property and that two or more branches force any equality kernel to be constant. Applying the same substitution to a KNRS graph yields a graph satisfying the KNRS inequality on~$\rho$-regular, $\rho$-locally dense graphons (\cref{thm:uniform-theta}). For substitutions that vary across edges, permutation averaging equalizes the number of branches of each length across all pairs of original vertices. Under the branch-count conditions of Im, Li, and Liu~\citep{ImLiLiu2026}, this gives Sidorenko-goodness for arbitrary branch lengths, and whenever the base graph has at least three vertices, equality at positive density forces the kernel to be constant (\cref{thm:nonuniform-sidgood-theta}).

\section{Comparison with an edge-transitive graph}\label{sec:holder}

The comparison depends only on how often each edge of~$J$ is covered by images of~$F$. We first average over automorphisms on an arbitrary graphon, and then apply local density.

\begin{fact}[H\"older's inequality]\label[fact]{fact:holder-averaging}
Let~$f_1,\ldots,f_k$ be nonnegative integrable functions on a probability space. For positive weights~$a_1,\ldots,a_k$ with~$\sum_{i=1}^k a_i=1$,
\begin{equation*}
    \int \prod_{i=1}^k f_i^{a_i}\le \prod_{i=1}^k\left(\int f_i\right)^{a_i}\,.
\end{equation*}
\end{fact}

In our applications the weights~$a_i$ are equal, and the functions~$f_i$ differ only by permutations of the variables and thus have the same integral.

\begin{lemma}[H\"older comparison]\label[lemma]{lem:holder-comparison}
Let~$F$ be a loopless multigraph and~$J$ be a simple edge-transitive graph on the same vertex set, with~$e(F)>0$ and~$\supp(F)\subseteq J$. For every graphon~$U$,
\begin{equation*}
    t(F,U)\ge t\bigl(J,U^{\circ e(F)/e(J)}\bigr)\,.
\end{equation*}
\end{lemma}

\begin{proof}
Choose a subgroup~$\Gamma\le\Aut(J)$ acting transitively on~$E(J)$. For~$\gamma\in\Gamma$, set~$f_\gamma(\mathbf x)=\prod_{uv\in E(F)}U(x_{\gamma(u)},x_{\gamma(v)})$, counting parallel edges with multiplicity. Relabelling the variables gives~$\int f_\gamma=t(F,U)$. By~\cref{fact:holder-averaging} with uniform weights~$1/|\Gamma|$,
\begin{equation}\label{eq:holder-average-copies}
    t(F,U)\ge \int \prod_{\gamma\in\Gamma} f_\gamma(\mathbf x)^{1/|\Gamma|}\dd\mathbf x\,.
\end{equation}
For each edge~$ab\in E(J)$, its exponent in the geometric mean is the average multiplicity of~$ab$ in~$\gamma(F)$ over~$\gamma\in\Gamma$. By edge transitivity, this average is the same for every edge of~$J$, and the averages sum to~$e(F)$, so each edge receives exponent~$e(F)/e(J)$. Substituting this into~\eqref{eq:holder-average-copies} gives the claim.
\end{proof}

To analyze equality strictly above the comparison exponent, we next characterize when an edge product can be constant.

\begin{lemma}\label[lemma]{lem:edge-product-constant}
Let~$G$ be a loopless multigraph with at least one edge, and let~$W\colon[0,1]^2\to[0,\infty)$ be symmetric and measurable. If~$W_G$ is almost everywhere equal to a positive constant, then~$W$ is constant almost everywhere.
\end{lemma}

\begin{proof}
Since~$W_G>0$ almost everywhere and~$G$ has at least one edge, integrating out the remaining variables shows that~$W>0$ almost everywhere. Fix a vertex~$v$ of positive degree, let~$u_1,\ldots,u_d$ be its distinct neighbours, and let~$m_i$ be the multiplicity of~$vu_i$. Evaluating~$W_G$ at two tuples that agree outside~$v$ and take values~$x$ and~$x'$ at~$v$, Fubini's theorem allows us to cancel all factors not incident with~$v$. Thus, for almost every~$(x,x')$, the ratio~$R(z)=W(x,z)/W(x',z)$ satisfies
\begin{equation*}
    \prod_{i=1}^d R(z_i)^{m_i}=1
    \qquad\text{for almost every }(z_1,\ldots,z_d)\,.
\end{equation*}
Fixing generic values of~$z_2,\ldots,z_d$ shows that~$R(z_1)^{m_1}$---and hence~$R(z_1)$---is constant almost everywhere. Its positive constant value~$c$ satisfies~$c^{\sum_i m_i}=1$, so~$c=1$, meaning~$W(x,z)=W(x',z)$ for almost every~$z$. Fixing a typical~$x'$ gives~$W(x,z)=g(z)$ almost everywhere for some function~$g$. Symmetry of~$W$ then yields~$g(x)=g(z)$ almost everywhere.
\end{proof}

\begin{proof}[Proof of \cref{thm:comparison-intro}]
Write~$m=e(F)$, $M=e(J)$, and~$\alpha=qm/M\ge1$. Applying~\cref{lem:holder-comparison} to~$U=W^{\circ q}$ and then Jensen's inequality on~$[0,1]^{V(J)}$ gives
\begin{equation*}
    t(F,W^{\circ q})\ge \int W_J^\alpha
    \ge \left(\int W_J\right)^\alpha
    =t(J,W)^\alpha\ge \rho^{M\alpha}=\rho^{qm}\,,
\end{equation*}
since~$J$ is KNRS.

Conversely, the constant graphon~$W=\rho$ clearly gives equality. Suppose~$0<\rho<1$ and equality holds. Then~$t(J,W)=\rho^M$, which forces~$W=\rho$ almost everywhere if~$J$ is KNRS-forcing. If~$\alpha>1$, strict convexity in Jensen's inequality also forces~$W_J=\rho^M$ almost everywhere. By~\cref{lem:edge-product-constant}, $W$ is constant almost everywhere, and~$t(J,W)=\rho^M$ fixes its value at~$\rho$.
\end{proof}

\begin{proof}[Proof of \cref{thm:main}]
Delete the isolated vertices of~$F$ and apply~\cref{thm:comparison-intro} with~$J=K_n$. The clique~$K_n$ is KNRS~\citep{KNRS2010}, and~$S_n$ acts transitively on its~$\binom n2$ edges.
\end{proof}

\begin{corollary}\label[corollary]{cor:bipartite}
Let~$F$ be bipartite with parts of sizes~$a$ and~$b$ and with~$m=e(F)>0$. If~$q\ge ab/m$, then~$t(F,W^{\circ q})\ge\rho^{qm}$ for every~$\rho$-locally dense graphon~$W$.
\end{corollary}

\begin{proof}
Apply~\cref{thm:comparison-intro} with~$J=K_{a,b}$, which is KNRS~\citep{KNRS2010} and edge-transitive under permutations preserving its two parts.
\end{proof}

For an integer~$d\ge1$, let~$dH$ denote the multigraph obtained by replacing each edge of~$H$ by~$d$ parallel edges, so that~$t(dH,W)=\int W_H^d$. Conlon and Lee proved that every bipartite graph has a bipartite companion whose disjoint union with it is Sidorenko~\citep[Corollary~1.4]{ConlonLee2021}. Parallel edges give an analogous statement for KNRS.

\begin{corollary}\label[corollary]{cor:parallel-thickening}
For every graph~$H$, there is a multigraph~$H'$ such that the disjoint union~$H\sqcup H'$ is KNRS.
\end{corollary}

\begin{proof}
If~$H$ has no edges, take~$H'$ empty. Otherwise let~$m=e(H)$, let~$n$ be the number of non-isolated vertices of~$H$, and choose an integer~$d\ge1$ such that~$q=(d+1)/2\ge\binom n2/m$. Take~$H'$ to be a disjoint copy of~$dH$. For any~$\rho$-locally dense graphon~$W$, put~$f=W_H$; Cauchy--Schwarz and~\cref{thm:comparison-intro} with~$J=K_n$ (after deleting isolated vertices) give
\begin{equation*}
    t(H\sqcup H',W)=\left(\int f\right)\left(\int f^d\right)
    \ge \left(\int f^q\right)^2\ge \rho^{2qm}\,.
\end{equation*}
Since~$H\sqcup H'$ has~$(d+1)m=2qm$ edges, $H\sqcup H'$ is KNRS.
\end{proof}

\section{The \texorpdfstring{$L^p$}{Lp}-KNRS exponent}\label{sec:exponent}

Fix a simple graph~$F$ with~$m=e(F)>0$. Call~$q\ge1$ \emph{admissible} for~$F$ if~$\|W_F\|_q\ge\rho^m$ for every~$\rho\in[0,1]$ and every~$\rho$-locally dense graphon~$W$. Define
\begin{equation*}
    \kappa_{\mathrm{KNRS}}(F)=\inf\{q\ge1:q\text{ is admissible for }F\}\,.
\end{equation*}
By~\cref{thm:main}, $\kappa_{\mathrm{KNRS}}(F)\le\binom n2/m$, where~$n$ counts the non-isolated vertices of~$F$. More generally, any edge-transitive KNRS supergraph~$J$ as in~\cref{thm:comparison-intro} yields~$\kappa_{\mathrm{KNRS}}(F)\le e(J)/m$.

\begin{proposition}\label[proposition]{prop:ray}
The admissible exponents for~$F$ form the closed ray~$[\kappa_{\mathrm{KNRS}}(F),\infty)$. For every~$q>\kappa_{\mathrm{KNRS}}(F)$ and~$0<\rho<1$, equality~$\|W_F\|_q=\rho^m$ for a~$\rho$-locally dense graphon~$W$ holds only if~$W=\rho$ almost everywhere.
\end{proposition}

\begin{proof}
Since~$[0,1]^{V(F)}$ is a probability space, $\|W_F\|_q$ is nondecreasing in~$q$, so every exponent above an admissible exponent is admissible. For a fixed graphon~$W$, since~$0\le W_F\le1$, dominated convergence implies that~$q\mapsto\|W_F\|_q$ is continuous on~$[1,\infty)$. Letting admissible exponents tend to~$k=\kappa_{\mathrm{KNRS}}(F)$ shows that~$k$ is itself admissible.

Suppose~$q>k$, $0<\rho<1$, and~$\|W_F\|_q=\rho^m$. Then~$\rho^m\le\|W_F\|_k\le\|W_F\|_q=\rho^m$. Equality in Jensen's inequality for~$u\mapsto u^{q/k}$ forces~$W_F^k$ to be constant almost everywhere. Its value~$\rho^{km}$ is positive, so~\cref{lem:edge-product-constant} gives~$W=\rho$ almost everywhere.
\end{proof}

\begin{proposition}\label[proposition]{prop:endpoint}
A simple graph~$F$ with at least one edge is KNRS if and only if~$1$ is admissible, or equivalently if and only if~$\kappa_{\mathrm{KNRS}}(F)=1$.
\end{proposition}

\begin{proof}
The first equivalence holds since~$\|W_F\|_1=t(F,W)$, and the second follows from~\cref{prop:ray}.
\end{proof}

There is also a local form of the moment comparison, showing how the local density parameter behaves under pointwise powers.

\begin{lemma}\label[lemma]{lem:power-locally-dense}
If~$W$ is~$\rho$-locally dense and~$\alpha\ge1$, then~$W^{\circ\alpha}$ is~$\rho^\alpha$-locally dense. If~$\alpha>1$, $0<\rho<1$, and~$\int W^\alpha=\rho^\alpha$, then~$W=\rho$ almost everywhere.
\end{lemma}

\begin{proof}
For any measurable set~$S$ of positive measure, Jensen's inequality with respect to normalized measure on~$S^2$ gives
\begin{equation*}
    \frac{1}{\lambda(S)^2}\int_{S^2}W^\alpha
    \ge \left(\frac{1}{\lambda(S)^2}\int_{S^2}W\right)^\alpha\ge\rho^\alpha\,,
\end{equation*}
while the inequality holds trivially when~$\lambda(S)=0$. For~$S=[0,1]$, the assumption~$\int W^\alpha=\rho^\alpha$ forces equality in both steps; strict convexity then makes~$W$ constant, and the second equality fixes its value at~$\rho$.
\end{proof}

\section{Positive semidefinite kernels and subdivisions}\label{sec:preliminaries}

Replacing an edge by a path corresponds to taking an operator power, and replacing it by several internally disjoint paths corresponds to taking a pointwise product of operator powers. Both operations preserve positive semidefiniteness, which allows the branch lengths in a theta substitution to have arbitrary parity.

A \emph{kernel} is a bounded symmetric measurable function~$K\colon[0,1]^2\to\mathbb R$, with integral operator~$T_K f(x)=\int_0^1 K(x,y)f(y)\dd y$ on real~$L^2[0,1]$. The edge-product formula defining~$t(F,W)$ extends directly to~$t(F,K)$. We write~$K\succeq0$ if~$\langle f,T_Kf\rangle\ge0$ for all~$f\in L^2[0,1]$, and call~$K$ \emph{doubly nonnegative} if~$K\succeq0$ and~$K\ge0$ almost everywhere. For~$K\ge0$, write~$\|K\|_1=\int K$. We use~$\1$ for the constant function or kernel of value~$1$, so~$T_{\1}$ is the orthogonal projection onto the constants.

For an integer~$r\ge1$, let~$K^{[r]}$ denote the kernel of~$T_K^r$, given recursively by~$K^{[1]}=K$ and
\begin{equation*}
    K^{[r+1]}(x,y)=\int_0^1 K^{[r]}(x,z)K(z,y)\dd z\,.
\end{equation*}
Each~$K^{[r]}$ is bounded, and~$K\ge0$ implies~$K^{[r]}\ge0$. By the spectral theorem, $K\succeq0$ implies~$K^{[r]}\succeq0$ for every~$r\ge1$. By the Schur product theorem for bounded measurable kernels~\citep[Lemma~2.17]{BradacSudakovWigderson2024}, any finite pointwise product of positive semidefinite kernels is positive semidefinite; testing its operator against~$\1$ shows in particular that its integral is nonnegative.

\subsection{Local density and the one-subdivision}\label{app:relations}

A graphon~$W$ is \emph{$\rho$-regular} if~$T_W\1=\rho\1$. A graph~$G$ is \emph{regular-KNRS} if~$t(G,W)\ge\rho^{e(G)}$ for every~$\rho\in[0,1]$ and every~$\rho$-locally dense, $\rho$-regular graphon~$W$. It is \emph{regular-KNRS-forcing} if, in addition, equality for~$0<\rho<1$ implies~$W=\rho$ almost everywhere.

Following Sidorenko~\citep{SidorenkoDNN2021}, we call~$G$ \emph{Sidorenko-good} (simply \emph{good} in~\citep{SidorenkoDNN2021}) if
\begin{equation*}
    t(G,K)\ge \|K\|_1^{e(G)}
    \qquad\text{for every bounded doubly nonnegative kernel }K\,.
\end{equation*}
A Sidorenko-good graph is \emph{Sidorenko-good-forcing} if equality for~$\|K\|_1>0$ implies~$K=\|K\|_1$ almost everywhere. Parallel edges are counted with multiplicity.

Sidorenko's estimates in~\citep{SidorenkoDNN2021} are stated for kernels admitting bounded Gram factors. To apply them to a general bounded doubly nonnegative kernel~$K$, partition~$[0,1]$ into~$d=2^n$ equal intervals~$I_1,\ldots,I_d$ ($n\ge1$), and let~$U_n$ equal the average of~$K$ on each block~$I_i\times I_j$ plus~$1/n$. Its block-value matrix~$A$ is positive semidefinite: for~$c\in\mathbb R^d$ and~$f=\sum_i c_i\mathbf{1}_{I_i}$, we have~$c^\top A c=d^2\langle f,T_K f\rangle+(\sum_i c_i)^2/n\ge0$. Defining~$R_n$ to take value~$\sqrt d\,(A^{1/2})_{ij}$ on~$I_i\times I_j$ gives a bounded kernel with~$U_n(x,y)=\int_0^1 R_n(x,z)R_n(y,z)\dd z$ almost everywhere.

The kernels~$U_n$ satisfy~$1/n\le U_n\le M:=\|K\|_\infty+1$ and converge to~$K$ in~$L^1$ by convergence of dyadic conditional expectations. For a loopless multigraph~$G$ with~$e=e(G)>0$, telescoping the edge product yields
\begin{equation}\label{eq:step-density-continuity}
    |t(G,U_n)-t(G,K)|\le e M^{e-1}\|U_n-K\|_1\,,
\end{equation}
since each term in the telescoping sum has only one edge carrying~$U_n-K$, and integrating its absolute value over its two distinct endpoints yields~$\|U_n-K\|_1$ even in the presence of parallel edges.

A kernel~$B$ is \emph{copositive} if~$\langle f,T_Bf\rangle\ge0$ for every bounded nonnegative~$f$. Being~$\rho$-locally dense is equivalent to copositivity of~$W-\rho\1$~\citep[Lemma~2.13]{BradacSudakovWigderson2024}, while~$\rho$-regularity strengthens this to positive semidefiniteness.

Write~$\Sub(G)$ for the \emph{one-subdivision} of~$G$, obtained by inserting one new vertex into each edge. Conlon, Kim, Lee, and Lee~\citep{ConlonKimLeeLee2018} proved that~$\Sub(G)$ is Sidorenko when~$G$ is KNRS (with longer subdivisions treated by Chen, Lin, and Ma~\citep{ChenLinMa2024}), and Sidorenko~\citep[Theorem~3.5]{SidorenkoDNN2021} proved the same when~$G$ is Sidorenko-good. We record the role of regular local density in both implications.

\begin{proposition}\label[proposition]{prop:notion-relations}
Every bipartite Sidorenko graph is KNRS. Every KNRS graph and every Sidorenko-good graph is regular-KNRS. If~$G$ is regular-KNRS, then~$\Sub(G)$ is Sidorenko.
\end{proposition}

\begin{proof}
A~$\rho$-locally dense graphon~$W$ satisfies~$t(K_2,W)\ge\rho$ (by taking~$S=[0,1]$), so the Sidorenko inequality implies the KNRS inequality; restricting to~$\rho$-regular graphons gives the regular-KNRS property.

If~$W$ is both~$\rho$-locally dense and~$\rho$-regular, then~$W-\rho\1\succeq0$ by Brada\v{c}, Sudakov, and Wigderson~\citep[Corollary~2.15]{BradacSudakovWigderson2024}. Thus~$W$ is doubly nonnegative with~$\|W\|_1=\rho$, so any Sidorenko-good graph satisfies~$t(G,W)\ge\|W\|_1^{e(G)}=\rho^{e(G)}$ and is regular-KNRS.

For the last assertion, $\Sub(G)$ is bipartite with all inserted vertices in one part. By the regular reduction for Sidorenko's inequality~\citep[Lemma~2.2]{ImLiLiu2026}, it suffices to consider a~$p$-regular graphon~$W$. Put~$U=W^{[2]}$. Then~$T_U\1=p^2\1$, and for every measurable~$S$, Cauchy--Schwarz gives
\begin{equation*}
    \int_{S^2} U
    =\int_0^1\left(\int_S W(x,z)\dd x\right)^2\dd z
    \ge \left(\int_S\int_0^1 W(x,z)\dd z\dd x\right)^2
    =p^2\lambda(S)^2\,.
\end{equation*}
Thus~$U$ is~$p^2$-regular and~$p^2$-locally dense. Integrating over the new subdivision vertices and using regular-KNRS gives~$t(\Sub(G),W)=t(G,U)\ge p^{2e(G)}$, the Sidorenko bound for~$\Sub(G)$.
\end{proof}

\section{Theta substitutions}\label{sec:good-theta}

A \emph{rooted theta pattern}~$\Theta=P(r_1,\ldots,r_s)$ consists of~$s\ge1$ internally disjoint paths between two distinct roots, with positive integer lengths~$r_1,\ldots,r_s$. We write~$P_r$ for the path of length~$r$. Allowing multiple branches of length~$1$ (parallel edges) will let us halve the lengths of all branches of an even theta graph, including branches of length~$2$.

Given a graph~$H$, let~$H[\Theta]$ be obtained by replacing each edge by an independent copy of~$\Theta$ with its roots at the endpoints of the edge. For a bounded nonnegative kernel~$K$, integrating over the internal vertices of one copy of~$\Theta$ yields the rooted kernel
\begin{equation*}
    K_\Theta^*(x,y)=\prod_{i=1}^s K^{[r_i]}(x,y)\,.
\end{equation*}
Because the branches have disjoint interiors, their contributions multiply when the roots are fixed. Integrating over the roots of~$\Theta$, or over all original vertices of~$H$, gives respectively
\begin{align}
    \|K_\Theta^*\|_1 &= t(\Theta,K)\,,\label{eq:theta-edge-density}\\
    t(H[\Theta],K) &= t(H,K_\Theta^*)\,.\label{eq:rooted-subdivision}
\end{align}
In particular, $K_\Theta^*$ is doubly nonnegative whenever~$K$ is, and is a graphon whenever~$K$ is.

\subsection{Equality for a theta pattern}

Sidorenko proved a weighted good-graph inequality for simple theta graphs~\citep[Theorem~6.5]{SidorenkoDNN2021}, and Brada\v{c}, Sudakov, and Wigderson showed that generalized theta graphs are regular-KNRS~\citep[Theorem~1.9]{BradacSudakovWigderson2024}. When there are at least two branches, any positive eigenvalue on~$\1^\perp$ produces a strict excess in~$t(\Theta,K)$.

\begin{lemma}\label[lemma]{lem:theta-forcing}
A rooted theta pattern with at least two branches is Sidorenko-good-forcing and regular-KNRS-forcing.
\end{lemma}

\begin{proof}
Let~$\Theta=P(r_1,\ldots,r_s)$ with~$s\ge2$ and~$m=\sum_{i=1}^s r_i$. Let~$K$ be doubly nonnegative and set~$p=\|K\|_1$. If~$p=0$, then~$K=0$ almost everywhere; assume~$p>0$. The compact positive semidefinite operator~$T_K$ has largest eigenvalue~$\lambda\ge\langle\1,T_K\1\rangle=p$. Take a real unit eigenfunction~$\phi$ for~$\lambda$. Since~$K\ge0$, replacing~$\phi$ by~$|\phi|$ cannot decrease the Rayleigh quotient~$\langle\phi,T_K\phi\rangle$; by maximality of~$\lambda$, $|\phi|$ is also an eigenfunction, so we may assume~$\phi\ge0$. Moreover, $\phi=\lambda^{-1}T_K\phi$ is bounded because~$K$ is bounded.

Writing~$(\phi\otimes\phi)(x,y)=\phi(x)\phi(y)$, the spectral theorem gives a bounded positive semidefinite remainder
\begin{equation*}
    B_r=K^{[r]}-\lambda^r\phi\otimes\phi\succeq0\qquad(r\ge1)\,.
\end{equation*}
Substituting~$K^{[r_i]}=\lambda^{r_i}\phi\otimes\phi+B_{r_i}$ into the pointwise product~$K_\Theta^*$ and expanding, every term is a product of positive semidefinite kernels and therefore has nonnegative integral. Keeping only the term with~$\lambda^{r_i}\phi\otimes\phi$ on every branch gives
\begin{equation*}
    t(\Theta,K)\ge\lambda^m\left(\int_0^1\phi^s\right)^2\ge\lambda^m\ge p^m\,,
\end{equation*}
where~$\int_0^1\phi^s\ge1$ holds by H\"older's inequality since~$\phi\ge0$, $\|\phi\|_2=1$, and~$s\ge2$.

If equality holds, then~$\lambda=p$, so~$\1$ attains the maximal Rayleigh quotient of~$T_K$ and hence~$T_K\1=p\1$. Suppose~$T_K$ has a unit eigenfunction~$\psi\perp\1$ with eigenvalue~$\mu>0$; as above, $\psi$ is bounded. Since~$\1\perp\psi$ are eigenfunctions with eigenvalues~$p$ and~$\mu$, in each factor~$K^{[r]}$ we can split off~$p^r\1$ and~$\mu^r\psi\otimes\psi$, leaving a positive semidefinite remainder. In the expansion of~$K_\Theta^*$, the all-constant term contributes~$p^m$, while choosing~$\mu^{r_i}\psi\otimes\psi$ on branches~$1$ and~$2$ and~$p^{r_i}\1$ on all other branches contributes
\begin{equation*}
    \mu^{r_1+r_2}p^{m-r_1-r_2}\left(\int_0^1\psi^2\right)^2>0\,.
\end{equation*}
Every remaining term has nonnegative integral, contradicting~$t(\Theta,K)=p^m$. Thus~$T_K$ vanishes on~$\1^\perp$, so~$K=p\1$ almost everywhere. Conversely, constant kernels give equality. The regular-KNRS-forcing claim follows since every~$\rho$-regular, $\rho$-locally dense graphon is doubly nonnegative with integral~$\rho$ (\cref{prop:notion-relations}).
\end{proof}

\begin{remark}\label[remark]{rem:path-not-forcing}
A single path~$P_r$ is Sidorenko-good (which also follows from Sidorenko's result for trees~\citep[Corollary~6.4]{SidorenkoDNN2021}): if~$\nu$ is the spectral measure of~$T_K$ at~$\1$, then~$\nu$ is a probability measure on~$[0,\infty)$ and Jensen's inequality gives
\begin{equation*}
    t(P_r,K)=\langle\1,T_K^r\1\rangle
    =\int u^r\dd\nu(u)\ge\left(\int u\dd\nu(u)\right)^r=\|K\|_1^r\,.
\end{equation*}
However, a single path is not forcing: letting~$f$ be~$1$ on~$[0,1/2)$ and~$-1$ on~$[1/2,1]$, for~$0<\varepsilon<p$ the kernel~$K=p\1+\varepsilon f\otimes f$ is doubly nonnegative with~$T_K\1=p\1$, so~$t(P_r,K)=p^r$ for all~$r\ge1$.
\end{remark}

\subsection{Uniform substitutions}

For a Sidorenko-good base graph~$H$, the bound for~$H[\Theta]$ follows by applying goodness to~$K_\Theta^*$ and then to~$\Theta$; for a KNRS base graph, we verify that~$W_\Theta^*$ is locally dense.

\begin{theorem}[Uniform theta-substitutions]\label{thm:uniform-theta}
Let~$H$ be a graph with at least one edge, and let~$\Theta=P(r_1,\ldots,r_s)$ be a rooted theta pattern.
\begin{enumerate}
    \item If~$H$ is Sidorenko-good, then~$H[\Theta]$ is Sidorenko-good. If in addition~$H$ is Sidorenko-good-forcing or~$s\ge2$, then~$H[\Theta]$ is Sidorenko-good-forcing.
    \item If~$H$ is KNRS, then~$H[\Theta]$ is regular-KNRS. If~$H$ is KNRS-forcing, then~$H[\Theta]$ is regular-KNRS-forcing.
\end{enumerate}
\end{theorem}

When~$s=1$, Brada\v{c}, Sudakov, and Wigderson~\citep[Proposition~4.8]{BradacSudakovWigderson2024} obtained the regular-KNRS conclusion for every regular-KNRS base graph, and Kiem, Parczyk, and Spiegel~\citep[Proposition~1.4]{KiemParczykSpiegel2024} proved that uniform subdivisions preserve forcing families in the usual quasirandomness sense.

\begin{proof}
Write~$m=e(\Theta)$ and~$b=e(H)$. Suppose first that~$H$ is Sidorenko-good and~$K$ is doubly nonnegative. By~\eqref{eq:rooted-subdivision}, goodness of~$H$ applied to~$K_\Theta^*$, and~\eqref{eq:theta-edge-density},
\begin{equation*}
    t(H[\Theta],K)=t(H,K_\Theta^*)\ge \|K_\Theta^*\|_1^b
    =t(\Theta,K)^b\ge\|K\|_1^{mb}\,,
\end{equation*}
where the last step uses~\cref{lem:theta-forcing} for~$s\ge2$ and~\cref{rem:path-not-forcing} for~$s=1$. Since~$e(H[\Theta])=mb$, $H[\Theta]$ is Sidorenko-good. Let~$p=\|K\|_1>0$ and suppose equality holds. If~$s\ge2$, equality for~$\Theta$ forces~$K=p\1$ almost everywhere. If~$s=1$ and~$H$ is Sidorenko-good-forcing, equality for~$H$ forces~$K_\Theta^*=K^{[r_1]}=p^{r_1}\1$ almost everywhere. Since~$T_K$ is positive semidefinite and~$pT_{\1}$ is the unique positive semidefinite~$r_1$th root of~$p^{r_1}T_{\1}$, we get~$T_K=pT_{\1}$ and thus~$K=p\1$ almost everywhere.

Now let~$H$ be KNRS and~$W$ be a~$\rho$-locally dense, $\rho$-regular graphon. The case~$\rho=0$ is trivial, so assume~$\rho>0$. Set~$A=W-\rho\1$. Then~$A\succeq0$ and~$T_A\1=0$; since~$T_A$ is self-adjoint, both operator products of~$T_A$ with~$T_{\1}$ vanish, giving
\begin{equation}\label{eq:path-power-decomp}
    W^{[r]}=\rho^r\1+A^{[r]}\qquad(r\ge1)\,.
\end{equation}
Expanding the pointwise product~$W_\Theta^*=\prod_{i=1}^s W^{[r_i]}$ via~\eqref{eq:path-power-decomp} yields
\begin{equation*}
    W_\Theta^*-\rho^m\1
    =\sum_{\varnothing\ne I\subseteq[s]}
    \rho^{m-\sum_{i\in I}r_i}\prod_{i\in I} A^{[r_i]}\succeq0\,,
\end{equation*}
because each summand is positive semidefinite by the Schur product theorem. Testing against indicator functions shows that~$W_\Theta^*$ is~$\rho^m$-locally dense, so the KNRS property of~$H$ gives~$t(H[\Theta],W)=t(H,W_\Theta^*)\ge\rho^{mb}$.

Suppose further that~$H$ is KNRS-forcing, $0<\rho<1$, and equality holds. Then~$W_\Theta^*=\rho^m\1$ almost everywhere. In the expansion above, a sum of positive semidefinite operators with positive coefficients vanishes, so every summand is zero. Taking~$I=\{1\}$ gives~$T_A^{r_1}=0$, whence~$T_A=0$ by the spectral theorem and~$W=\rho\1$ almost everywhere.
\end{proof}

\subsection{Non-uniform substitutions}

Let~$H$ be a simple graph on~$[h]=\{1,\ldots,h\}$ with~$h\ge2$, and let~$F$ be obtained by replacing each edge~$ij$ with a rooted theta pattern that may depend on~$ij$. For each unordered pair~$ij$ and length~$r\ge1$, let~$h_{ij}(r)$ be the number of branches of length~$r$ between~$i$ and~$j$ (with~$h_{ij}(r)=0$ for non-edges). Write
\begin{equation*}
    C=\binom h2\,,\qquad N_r=\sum_{i<j}h_{ij}(r)\,,\qquad \alpha_r=N_r/C\,.
\end{equation*}
Thus~$N_r$ counts all branches of length~$r$, and~$\alpha_r$ is their average per vertex pair (including non-edges). When all~$\alpha_r$ are integers and at least one is positive, they define a single theta pattern~$\Theta_{\mathrm{av}}$ that can be placed uniformly on every edge of~$K_h$.

\begin{theorem}[Non-uniform theta-substitutions]\label{thm:nonuniform-sidgood-theta}
Suppose~$e(F)>0$. Then~$F$ is Sidorenko-good under either of the following hypotheses:
\begin{enumerate}
    \item $C$ divides~$N_r$ for every~$r\ge1$.
    \item All branches have one common length~$r$, and~$N_r\ge C$.
\end{enumerate}
Under either hypothesis, if~$h\ge3$, then~$F$ is Sidorenko-good-forcing; if~$h=2$, then~$F$ is Sidorenko-good-forcing if and only if~$F$ has at least two branches.
\end{theorem}

The permutation averaging below is that of Im, Li, and Liu~\citep[Lemma~3.4]{ImLiLiu2026} for even branches, extended here to arbitrary branch lengths and general nonnegative kernels.

\begin{lemma}\label[lemma]{lem:nonuniform-theta-holder}
With the notation above, every bounded nonnegative symmetric kernel~$K$ satisfies
\begin{equation*}
    t(F,K)\ge \int_{[0,1]^h}\prod_{i<j}\ \prod_{r:N_r>0}
        K^{[r]}(x_i,x_j)^{\alpha_r}\dd\mathbf x\,.
\end{equation*}
\end{lemma}

\begin{proof}
Integrating over the internal vertices of all branches leaves the integrand~$f(\mathbf x)=\prod_{i<j}\prod_{r\ge1} K^{[r]}(x_i,x_j)^{h_{ij}(r)}$ (omitting factors of exponent zero). Every relabelling of~$f$ by~$S_h$ has integral~$t(F,K)$. In their geometric mean over all~$h!$ permutations, for each fixed~$r$ the exponent of~$K^{[r]}(x_i,x_j)$ is the same for every pair~$ij$ because~$S_h$ acts transitively on unordered pairs; since these exponents sum to~$N_r$, each equals~$\alpha_r$. Applying~\cref{fact:holder-averaging} gives the bound.
\end{proof}

To handle equality when~$h\ge3$---including when the averaged pattern consists of a single path---we use the fact that~$K_h$ is Sidorenko-good-forcing, deduced from a comparison with~$K_3$ implicit in Sidorenko's proof for cliques.

\begin{lemma}\label[lemma]{lem:clique-forcing}
For every~$h\ge3$ and every bounded doubly nonnegative kernel~$K$,
\begin{equation*}
    t(K_h,K)\ge t(K_3,K)^{\binom h2/3}\,.
\end{equation*}
In particular, $K_h$ is Sidorenko-good-forcing.
\end{lemma}

\begin{proof}
The bound is an identity for~$h=3$ and holds trivially when~$p=\|K\|_1=0$. Assume~$h\ge4$ and~$p>0$. Set~$b=h-1$, $a=\binom b2$, $D=t(P_2,K)$, and~$T=t(K_3,K)$. Viewing~$K_h$ as~$K_b$ plus one universal vertex, we apply the H\"older estimate in Sidorenko's proof of Proposition~5.6~\citep{SidorenkoDNN2021} with base graph~$K_b$ (which is good by his Corollary~5.2) and auxiliary graphs~$K_{2,1}=P_2$ and~$K_3$. This yields the first inequality below for the positive step kernels~$U_n$ of~\cref{app:relations}, while his Lemma~5.5 gives the second. Passing to the limit via~\eqref{eq:step-density-continuity} gives
\begin{equation*}
    t(K_h,K)\,D^{a-b/2}\ge T^a\,,
    \qquad D^{3/2}\le T\,.
\end{equation*}
Since~$D=\int (T_K\1)^2\ge p^2>0$ and~$a-b/2\ge0$, combining these gives
\begin{equation*}
    t(K_h,K)\ge T^{a-(2a-b)/3}=T^{\binom h2/3}\,.
\end{equation*}
Finally, $K_3$ is the two-branch pattern~$P(1,2)$, so~\cref{lem:theta-forcing} gives~$T\ge p^3$ with equality for~$p>0$ only if~$K=p\1$ almost everywhere. The comparison transfers this to~$K_h$.
\end{proof}

\begin{proof}[Proof of \cref{thm:nonuniform-sidgood-theta}]
Fix a doubly nonnegative kernel~$K$ and put~$p=\|K\|_1$. For~$p=0$ we have~$K=0$ almost everywhere, so assume~$p>0$. If~$h=2$, then~$F$ is a single theta pattern, and the claims follow from~\cref{lem:theta-forcing,rem:path-not-forcing}. Assume~$h\ge3$, so~$K_h$ is Sidorenko-good-forcing by~\cref{lem:clique-forcing}.

In case~(1), let~$\Theta_{\mathrm{av}}$ have~$\alpha_r$ branches of length~$r$ for each~$r$. Put~$L=K_{\Theta_{\mathrm{av}}}^*$. The integrand in~\cref{lem:nonuniform-theta-holder} is~$\prod_{1\le i<j\le h}L(x_i,x_j)$, whose integral is~$t(K_h,L)=t(K_h[\Theta_{\mathrm{av}}],K)$. Thus~\cref{lem:nonuniform-theta-holder} and~\cref{thm:uniform-theta} give
\begin{equation*}
    t(F,K)\ge t(K_h[\Theta_{\mathrm{av}}],K)
    \ge p^{C e(\Theta_{\mathrm{av}})}=p^{e(F)}\,,
\end{equation*}
since~$C\sum_r r\alpha_r=\sum_r rN_r=e(F)$. Because~$K_h$ is Sidorenko-good-forcing, \cref{thm:uniform-theta} forces~$K=p\1$ almost everywhere whenever equality holds, even if~$\Theta_{\mathrm{av}}$ is a single path.

In case~(2), set~$\alpha=N_r/C\ge1$ and~$g(\mathbf x)=\prod_{i<j}K^{[r]}(x_i,x_j)$, so~$\int g=t(K_h[P_r],K)$. By~\cref{lem:nonuniform-theta-holder}, Jensen's inequality, and~\cref{thm:uniform-theta},
\begin{equation*}
    t(F,K)\ge \int g^\alpha\ge\left(\int g\right)^\alpha
    \ge p^{rC\alpha}=p^{e(F)}\,.
\end{equation*}
Since~$K_h$ is Sidorenko-good-forcing, so is~$K_h[P_r]$ by~\cref{thm:uniform-theta}, and equality again forces~$K=p\1$ almost everywhere. Conversely, constant kernels give equality in both cases.
\end{proof}

\paragraph{Even theta substitutions.}
Taking the one-subdivision of a theta substitution doubles each branch length; in particular, $\Sub(H[\Theta])=H[\Sub(\Theta)]$. If~$H$ is KNRS, then~$H[\Theta]$ is regular-KNRS by~\cref{thm:uniform-theta}, and its one-subdivision is Sidorenko by~\cref{prop:notion-relations}. Halving the branch lengths of any uniform even theta substitution and applying this argument recovers the uniform theorem of Im, Li, and Liu~\citep[Theorem~1.3]{ImLiLiu2026}. In the non-uniform case, halving preserves the branch counts (mapping~$N_{2r}$ to the count at length~$r$), so~\cref{thm:nonuniform-sidgood-theta} together with~\cref{prop:notion-relations} recovers their non-uniform theorem~\citep[Theorem~1.4]{ImLiLiu2026}.

\paragraph{AI assistance.} ChatGPT was used to assist with parts of the proofs and with the exposition. Gemini was used to assist with the exposition.
\bibliographystyle{plainnat}
{\small
\setlength{\bibsep}{0.15\baselineskip}
\bibliography{ref}

@misc{BradacSudakovWigderson2024,
  author       = {Brada{\v{c}}, Domagoj and Sudakov, Benny and Wigderson, Yuval},
  title        = {Counting subgraphs in locally dense graphs},
  howpublished = {arXiv:2406.12418},
  year         = {2024},
  note         = {To appear in the Israel Journal of Mathematics},
  url          = {https://arxiv.org/abs/2406.12418}
}

@misc{ChenLinMa2024,
  author       = {Chen, Hao and Lin, Yupeng and Ma, Jie},
  title        = {{Kohayakawa--Nagle--R{\"o}dl--Schacht} conjecture for subdivisions},
  howpublished = {arXiv:2407.10861},
  year         = {2024},
  note         = {To appear in the {SIAM} Journal on Discrete Mathematics},
  url          = {https://arxiv.org/abs/2407.10861}
}

@article{ConlonKimLeeLee2018,
  author  = {Conlon, David and Kim, Jeong Han and Lee, Choongbum and Lee, Joonkyung},
  title   = {Some advances on {Sidorenko}'s conjecture},
  journal = {Journal of the London Mathematical Society},
  volume  = {98},
  number  = {3},
  pages   = {593--608},
  year    = {2018},
  doi     = {10.1112/jlms.12142}
}

@article{ConlonLee2021,
  author  = {Conlon, David and Lee, Joonkyung},
  title   = {{Sidorenko}'s conjecture for blow-ups},
  journal = {Discrete Analysis},
  year    = {2021},
  note    = {Paper No. 2, 13 pp.},
  doi     = {10.19086/da.21472}
}

@incollection{ErdosSimonovits1984,
  author    = {Erd{\H{o}}s, Paul and Simonovits, Mikl{\'o}s},
  title     = {Cube-supersaturated graphs and related problems},
  booktitle = {Progress in Graph Theory (Waterloo, Ont., 1982)},
  publisher = {Academic Press},
  address   = {Toronto},
  pages     = {203--218},
  year      = {1984}
}

@article{ImLiLiu2026,
  author  = {Im, Seonghyuk and Li, Ruonan and Liu, Hong},
  title   = {{Sidorenko}'s conjecture for subdivisions and theta substitutions},
  journal = {Combinatorics, Probability and Computing},
  volume  = {35},
  number  = {2},
  pages   = {269--279},
  year    = {2026},
  doi     = {10.1017/S0963548325100242}
}

@misc{KiemParczykSpiegel2024,
  author       = {Kiem, Aldo and Parczyk, Olaf and Spiegel, Christoph},
  title        = {Adjunctions, box products, and forcing families},
  howpublished = {arXiv:2412.12904v2},
  year         = {2026},
  url          = {https://arxiv.org/abs/2412.12904v2}
}

@article{KNRS2010,
  author  = {Kohayakawa, Yoshiharu and Nagle, Brendan and R{\"o}dl, Vojt{\v{e}}ch and Schacht, Mathias},
  title   = {Weak hypergraph regularity and linear hypergraphs},
  journal = {Journal of Combinatorial Theory, Series B},
  volume  = {100},
  number  = {2},
  pages   = {151--160},
  year    = {2010},
  doi     = {10.1016/j.jctb.2009.05.005}
}

@article{Lee2021,
  author  = {Lee, Joonkyung},
  title   = {On some graph densities in locally dense graphs},
  journal = {Random Structures \& Algorithms},
  volume  = {58},
  number  = {2},
  pages   = {322--344},
  year    = {2021},
  doi     = {10.1002/rsa.20974}
}

@article{Reiher2014,
  author  = {Reiher, Christian},
  title   = {Counting odd cycles in locally dense graphs},
  journal = {Journal of Combinatorial Theory, Series B},
  volume  = {105},
  pages   = {1--5},
  year    = {2014},
  doi     = {10.1016/j.jctb.2013.12.002}
}

@article{Sidorenko1993,
  author  = {Sidorenko, Alexander},
  title   = {A correlation inequality for bipartite graphs},
  journal = {Graphs and Combinatorics},
  volume  = {9},
  number  = {2},
  pages   = {201--204},
  year    = {1993},
  doi     = {10.1007/BF02988307}
}

@article{SidorenkoDNN2021,
  author  = {Sidorenko, Alexander},
  title   = {Inequalities for doubly nonnegative functions},
  journal = {The Electronic Journal of Combinatorics},
  volume  = {28},
  number  = {1},
  pages   = {\#P1.32},
  year    = {2021},
  doi     = {10.37236/8947}
}

@misc{OpenAI2026,
  author       = {{OpenAI}},
  title        = {A counterexample to {Sidorenko}'s conjecture},
  howpublished = {\href{https://github.com/openai/math/tree/main/preprints/A-counterexample-to-Sidorenkos-conjecture-September-23-2026}{Research preprint}},
  month        = sep,
  year         = {2026}
}
}
\end{document}